\documentclass[12pt,a4paper,reqno]{amsart}
\usepackage{ amssymb, amsmath, enumerate, amsfonts, amsthm, mathrsfs, url, bm, mathtools}
\usepackage{amsmath}
\allowdisplaybreaks[4]
\usepackage{amsfonts}
\usepackage{bm}
\usepackage{soul}
\usepackage{ulem}
\usepackage{amssymb,amsthm,amsfonts,amsthm,latexsym,enumerate,url,cases}
\usepackage{enumerate, latexsym, amsmath, amsfonts, amssymb, amsthm, arydshln, color, mathrsfs}
\usepackage{tikz,xcolor,hyperref}
\numberwithin{equation}{section}
\usepackage{mathrsfs}
\usepackage{extarrows}
\definecolor{lime}{HTML}{A6CE39}
\DeclareRobustCommand{\orcidicon}{
	\begin{tikzpicture}
	\draw[lime, fill=lime] (0,0)
	circle [radius=0.16]
	node[white] {{\fontfamily{qag}\selectfont \tiny ID}};
	\draw[white, fill=white] (-0.0625,0.095)
	circle [radius=0.007];
	\end{tikzpicture}
	\hspace{-2mm}
}
\foreach \x in {A, ..., Z}{\expandafter\xdef\csname orcid\x\endcsname{\noexpand\href{https://orcid.org/\csname orcidauthor\x\endcsname}
			{\noexpand\orcidicon}}
}

\def\Z{{\mathbb Z}}

\def\N{{\mathbb N}}

\def\pmod #1{\ ({\rm{mod}}\ #1)}

\def\Ack{\medskip\noindent {\bf Acknowledgments}}

\theoremstyle{plain}
\newtheorem{theorem}{Theorem}
\newtheorem{lemma}{Lemma}
\newtheorem{corollary}{Corollary}

\newtheorem{conjecture}{Conjecture}
\theoremstyle{definition}

\usepackage{etoolbox}
\makeatletter
\patchcmd{\@settitle}{\uppercasenonmath\@title}{}{}{}
\patchcmd{\@setauthors}{\MakeUppercase}{}{}{}
\patchcmd{\section}{\scshape}{}{}{}
\makeatother

\begin{document}

\title
[{Representations of positive integers by three almost-prime squares}]
{Representations of positive integers by three almost-prime squares}

\author
[Yue-Feng She, Yu-Chen Sun, Guang-Liang Zhou]
{Yue-Feng She, Yu-Chen Sun, Guang-Liang Zhou}

\address {(Yue-Feng She) Department of Applied Mathematics, Nanjing Forestry University, Nanjing 210037, People's Republic of China}
\email{\tt she.math@njfu.edu.cn}

\address {(Yu-Chen Sun) School of Mathematics, University of Bristol, Bristol, BS8 1UG, England, United Kingdom}
\email{\tt yuchensun93@163.com}

\address{(Guang-Liang Zhou) Department of Applied Mathematics, Nanjing Forestry University, Nanjing 210037, People's Republic of China}
\email{guangliangzhou@126.com}

\keywords{Sum of three squares; almost primes; sieve methods.}
\subjclass[2020]{Primary 11E25; Secondary 11N35, 11N36.}
%\thanks{*Corresponding author}

\begin{abstract}
Let $P_r$ denote an integer with at most $r$ prime factors, counted with multiplicity. It is known that
every sufficiently large integer $N$ satisfying 
$N \equiv 3 \pmod{24}$ and $5 \nmid N$,
can be written in the form $N= x_1^2+x_2^2+x_3^2$ where $x_1,x_2,x_3$ are integers.
In this paper, we prove that the above representation in the following two different forms
(i) $x_1x_2x_3$ is a $P_{67}$-number; (ii) each $x_i$ is a $P_{27}$-number.

This result improves on the previous result of Waibel\cite{Wa}, in which $P_{72}$ was obtained in place of $P_{67}$. The proofs combine the higher-dimensional sieve, a Richert-type weighted sieve method introduced by Cai \cite{Cai} with a Bombieri-Vinogradov type result given by Waibel\cite{Wa}. 
Applying the same method in a one dimensional sieve setting, we also show that 
every sufficiently large $N$ not of the form $4^k(8l+7)$ can be written in the form
\[
N = x^{2} + y^{2} + (2^{a} z)^{2},
\]
where $x,y,a,z$ are non-negative integers and $z$ is a $P_{18}$-number. This improves upon a result of  Banerjee \cite{Ban} who obtained $P_{118}$ in place of $P_{18}$.

\end{abstract}
\maketitle

\section{Introduction}
\setcounter{lemma}{0}
\setcounter{theorem}{0}
\setcounter{corollary}{0}
\setcounter{remark}{0}
\setcounter{equation}{0}

 The representation of integers as sums of three squares has a long history. Legendre (1797) and Gauss (1796--1801) independently showed that an integer $m$ can be expressed as a sum of three non-negative squares if and only if $m$ is not of the form $4^k(8\ell+7)\ (k,\ell\in\mathbb{N}=\{0,1,2,\dots\})$. Moreover, Gauss gave an explicit formula for the representation number 
 \[
 r_3(m)=\#\{(x,y,z) \in \N^3: m=x^2+y^2+z^2\}
 \]
in terms of Hurwitz class numbers. 
By Siegel's lower bound for the class numbers of imaginary quadratic fields \cite{Si}, it follows that for any $\epsilon > 0,$
$$r_3(m)\gg m^{1/2 - \epsilon}.$$

A well-known conjecture related to the three-squares theorem is that every positive integer $n$ satisfying
\begin{align}\label{e35}
n \equiv 3 \pmod{24} \quad \text{and} \quad 5 \nmid n
\end{align}
can be represented as a sum of three squares of primes. Towards this conjecture, Blomer and Br\"udern \cite{BB} established that every sufficiently large integer $n$ satisfying (1.1) can be written as the sum of three squares of $P_r,$ where
\[
r = \begin{cases}
371, &\text{if}\ n \text{ is square-free}, \\
521, & \text{otherwise}.
\end{cases}
\]
This result has been improved by various authors \cite{Blo, Lv, Cai, Wa}.

In the present paper, we prove the following two theorems.
\begin{theorem}\label{thm1.2}
Let \(n\) be a sufficiently large integer satisfying condition \((\ref{e35})\).
Then $n$ can be represented as
\[
n=x_1^2+x_2^2+x_3^2,
\]
where $x_1x_2x_3$ is a $P_{67}$-number. In fact, the number of such representations is $\gg n^{1/2-\varepsilon}$ for any $\varepsilon>0$.
\end{theorem}

\begin{theorem}\label{thm1.3}
Let \(n\) be a sufficiently large integer satisfying condition \((\ref{e35})\).
Then $n$ can be represented as
\[
n=x_1^2+x_2^2+x_3^2,
\]
where each $x_i$ is a $P_{27}$-number for $1\le i\le 3$.  More precisely, the number of such representations is $\gg n^{1/2-\varepsilon}$ for any $\varepsilon>0$.
\end{theorem}

Theorem \ref{thm1.2} improves on a result of Waibel \cite{Wa}, where $P_{72}$ was obtained in place of $P_{67}$. Theorem \ref{thm1.3} is obtained from the same weighted-sieve argument, but the final almost-prime estimate is applied to each coordinate separately. The proof combines the three-dimensional sieve, a Richert-type weighted sieve method introduced by Cai \cite{Cai} with the level-of-distribution estimate arising from Waibel's work on ternary quadratic forms.

We next discuss a related application of the same method to the conjectures of Sun. Sun proposed the following two conjectures in \cite{Sun0} and \cite{Sun2} respectively.

\begin{conjecture}
Let $n$ be a positive integer. If $ n\equiv10\pmod{24} $, then $n$ can be expressed as
\[
n = x^{2} + y^{2} + (2^{a}3^{b})^{2},
\]
where $x,y,a,b$ are non-negative integers and $b>0$. If $ n\equiv5\pmod{12} $, then $n$ can be expressed as
\[
n = x^{2} + y^{2} + (2^{a}5^{b})^{2},
\]
where $x,y,a,b$ are non-negative integers and $a>0$.
\end{conjecture}

\begin{conjecture}
Every integer $n \geq 2$ can be expressed as
\[
n = x^{2} + y^{2} + (2^{a}3^{b})^{2} + (2^{c}5^{d})^{2},
\]
where \(x,y,a,b,c,d\) are non-negative integers. 
\end{conjecture}

Banerjee \cite{Ban} showed that every sufficiently large integer $m$ not of the form $4^k(8\ell + 7)$ for any $k,\ell\in\mathbb{N}$ can be represented as $ m = x^2 + y^2 + (2^a z)^2,$ such that $z$ is a $P_{118}$ and $x,y,a\in\mathbb{N}$. Using Waibel's diagonal norm estimate \cite[Theorem 2]{Wa} together with a one-dimensional weighted sieve, we improve this related result as follows.

\begin{theorem}\label{thm1.1}
Every sufficiently large integer $m$ not of the form $4^k(8\ell + 7)$ for any $k,\ell\in\mathbb{N}$ can be represented in the form
\[
m = x^2 + y^2 + (2^a z)^2,
\]
where $x,y,a,z$ are non-negative integers and $z$ is a $P_{18}$-number. Moreover, the number of such representations is $\gg m^{1/2-\varepsilon}$ for any $\varepsilon>0$.
\end{theorem}

\begin{corollary}\label{cor1.2}
Every sufficiently large integer $m$ can be represented in the form
\[
m = x^2 + y^2 + 2^{2a} + (2^b z)^2,
\]
where $x, y, a, b$ are non-negative integers and $z$ is a $P_{18}$-number.
\end{corollary}
\begin{proof}
Write $m=4^k m'$ with $4\nmid m'$. If $m'=1$, then $k>1$ for all sufficiently large $m$, and hence
\[
m=(2^{k-1})^2+(2^{k-1})^2+(2^{k-1})^2+(2^{k-1})^2,
\]
which has the required form. If $m'>1$, then one can choose $a\in\{0,1\}$ such that
\[
m'-4^a\not\equiv0,4,7\pmod 8
\qquad\text{and}\qquad
m'-4^a>0.
\]
For sufficiently large $m$, the integer $4^k(m'-4^a)$ is sufficiently large and is not of the form $4^u(8v+7)$. The result follows by applying Theorem \ref{thm1.1} to $4^k(m'-4^a)$.
\end{proof}

The paper is organized as follows. In Section 2 we collect the required estimates for theta series, the level of distribution supplied by Waibel's work, and the sieve lemmas used later. In Section 3 we apply a Richert-type weighted sieve in dimensions three and one. The three-dimensional argument proves Theorems \ref{thm1.2} and \ref{thm1.3}, while the one-dimensional argument proves Theorem \ref{thm1.1}.

\Ack\ The first author was supported by the Tianyuan Mathematical Foundation (Grant No. 12526613). The third author was supported by the National Natural Science Foundation of China (Grant No. 12401009). The second author would like to thank the number theory group at Nanjing Forestry University, especially Prof. Chen Wang, for their hospitality. Part of this work was completed while the second author was visiting Nanjing Forestry University.

\section{Notation and auxiliary lemmas}
\setcounter{lemma}{0}
\setcounter{theorem}{0}
\setcounter{corollary}{0}
\setcounter{remark}{0}
\setcounter{equation}{0}
In this paper, $m$ and $n$ represent sufficiently large positive integers, and $p$ denotes a prime number. $p^r \parallel n$ means that $p^r \mid n$ but $p^{r+1} \nmid n$. By convention, $\mu(n)$ is the M\"obius function, $\nu(n)$ is the number of distinct prime factors of $n$,
and $\Omega(n)$ denotes the total number of prime factors of $n$. 
Bold lowercase letters denote three-dimensional horizontal vectors, and their components are represented by the same letter with subscripts. For example, $\mathbf{y} = (y_1, y_2, y_3)$.

To successfully apply Waibel's mean-value result \cite[Lemma 17]{Wa}, we need to introduce the basic notations of quadratic forms and multiplicative structures.
Let
\[
Q=\begin{pmatrix}
2\ell_1^2&0&0\\
0&2\ell_2^2&0\\
0&0&2\ell_3^2
\end{pmatrix},
\]
where $\ell_i$ $(1\leq i\leq3)$ are positive integers. Define the associated ternary quadratic form
\[
f(\mathbf x)=\frac12\mathbf xQ\mathbf x^T.
\]
For an integer $n$, put
\[
r(f,n)=\#\{\mathbf x\in\mathbb Z^3:f(\mathbf x)=n\},
\]
and let
\[
\mathfrak o(f)=\#\{A\in SL_3(\mathbb Z):A^TQA=Q\}
\]
be the number of automorphs of $f$. We define the usual weighted means
\[
r(\mathrm{gen}\,f,n)=
\left(\sum_{\widetilde f\in\mathrm{gen}\,f}\frac1{\mathfrak o(\widetilde f)}\right)^{-1}
\sum_{\widetilde f\in\mathrm{gen}\,f}\frac{r(\widetilde f,n)}{\mathfrak o(\widetilde f)}
\]
and
\[
r(\mathrm{spn}\,f,n)=
\left(\sum_{\widetilde f\in\mathrm{spn}\,f}\frac1{\mathfrak o(\widetilde f)}\right)^{-1}
\sum_{\widetilde f\in\mathrm{spn}\,f}\frac{r(\widetilde f,n)}{\mathfrak o(\widetilde f)},
\]
where the summations are over representatives of the classes in the genus and spinor genus of $f$, respectively; see \cite[Section 102]{OM}.
For $n$ satisfying \eqref{e35}, Siegel's formula and the local-density computation used in \cite{Lv} give
\begin{equation}\label{e66}
r(\mathrm{gen}\,f,n)=\frac{\pi}{4}\frac{\lambda(\mathbf l,n)}{\ell_1\ell_2\ell_3}\mathfrak S(n)n^{1/2},
\end{equation}
where $\mathfrak S(n)\gg n^{-\varepsilon}$ and $\lambda(\mathbf l,n)$ is the same local factor as in \cite[(3.10)]{Lv}.
Let
    \[
    \mathcal{A} = \{ {\bf x} \in \mathbb{N}^3 : x_1^2 + x_2^2 + x_3^2 = n \},
    \]
    and for \( {\bf l} \in \mathbb{N}^3 \) with square-free odd components, define
    \begin{align*}
    \mathcal{A}_{\bf l} = \{ {\bf x} \in \mathcal{A} : x_j\equiv0\pmod{\ell_j}, j=1,2,3 \}.
    \end{align*}
The genus average in \eqref{e66} provides the main term for $|\mathcal A_{\mathbf l}|$; see \cite{BB} for details.
For a positive integer $t$, define
\[
\mathcal{A}(t) = \mathcal{A}(t, n) := \sum_{\substack{ x_1^2 + x_2^2 + x_3^2 = n \\ t | x_1 x_2 x_3}} 1.
\]
For square-free $t$ and $\mathbf x\in\mathcal A$, the elementary identity
\[
\mathbf 1_{t\mid x_1x_2x_3}
=\mu(t)\sum_{\substack{\mathbf l\\ [\ell_1,\ell_2,\ell_3]=t\\ \ell_j\mid x_j\ (1\leq j\leq3)}}
\mu(\ell_1)\mu(\ell_2)\mu(\ell_3)
\]
follows by multiplicativity, since it is immediate for $t=p$. Consequently,
\begin{equation}\label{e:A_t}
\mathcal A(t)=
\mu(t)\sum_{\substack{\mathbf l\\ [\ell_1,\ell_2,\ell_3]=t}}
\mu(\ell_1)\mu(\ell_2)\mu(\ell_3)|\mathcal A_{\mathbf l}|.
\end{equation}
Using \eqref{e66}, the expected main term for $\mathcal A(t)$ is
\[
M_t=X\mu(t)\sum_{\substack{\mathbf l\\ [\ell_1,\ell_2,\ell_3]=t}}
\mu(\ell_1)\mu(\ell_2)\mu(\ell_3)
\frac{\lambda(\mathbf l,n)}{\ell_1\ell_2\ell_3}
\qquad \text{and} \quad
X=\frac{\pi\mathfrak S(n)n^{1/2}}4.
\]
Let
\begin{equation}\label{e46}
W(t)=t\mu(t)\sum_{\substack{\mathbf l\\ [\ell_1,\ell_2,\ell_3]=t}}
\mu(\ell_1)\mu(\ell_2)\mu(\ell_3)
\frac{\lambda(\mathbf l,n)}{\ell_1\ell_2\ell_3}.
\end{equation}
Then
\[
M_t=\frac{W(t)}tX.
\]
In particular, $W(p)$ is the same function as $\Omega(p)$ in \cite[(3.11)]{Lv}. By \cite[(3.13)]{Lv}, we have $0\leq W(p)<p$, and there is a constant $L\geq2$ such that
\begin{equation}\label{e:V1}
V_1(w_1,w_2)^{-1}:=
\prod_{w_1<p\leq w_2}\left(1-\frac{W(p)}p\right)^{-1}
\leq
\left(\frac{\log w_2}{\log w_1}\right)^3
\left(1+\frac L{\log w_1}\right)
\end{equation}
for $2\leq w_1<w_2$. We also write $V_1(z)=V_1(2,z)$.

\begin{lemma}\label{pro1}
Let $D< n^{3/116}.$ Then we have
\begin{align*}
E(D):=\sum_{t\leq D}\mu^{2}(t)4^{\nu(t)}\Big|\mathcal{A}(t)-\frac{W(t)}{t} X\Big|\ll n^{1/2 - \varepsilon}.
\end{align*}
\end{lemma}

\noindent{\it Proof}. See \cite[Lemma 17]{Wa}.

We shall also need a one-dimensional version of the preceding setup. Let $m$ be an integer not of the form $4^k(8\ell+7)$, and take $\mathbf l=(1,1,d)$ with $d$ odd. For the form
\[
f_d(\mathbf x)=x_1^2+x_2^2+d^2x_3^2,
\]
and (\ref{e66}) gives
\begin{equation}\label{e67}
r(\mathrm{gen}\,f_d,m)=\frac{\pi}{4}\frac{\omega(d,m)}{d}\mathfrak S(m)m^{1/2},
\end{equation}
where $\omega(d,m)$ is a multiplicative function of $d$ satisfying $0\leq \omega(p,m)<p$ and
\begin{equation}\label{e71}
\prod_{w_1\leq p<w_2}\left(1-\frac{\omega(p,m)}p\right)^{-1}
\leq
\left(\frac{\log w_2}{\log w_1}\right)
\left(1+\frac L{\log w_1}\right)
\end{equation}
for $2\leq w_1<w_2$, where $L$ is an absolute constant. This is the one-dimensional analogue of \eqref{e66}.

Define
\[
\mathcal B=\{(x_1,x_2,x_3)\in\mathbb N_0^3:x_1^2+x_2^2+x_3^2=m\}
\]
and, for odd $d$,
\[
\mathcal B_d=\{(x_1,x_2,x_3)\in\mathcal B:x_3\equiv0\pmod d\}.
\]

\begin{lemma}\label{lem1}
Let $r(\mathrm{gen}\,f_d,m)$ be defined as before. For odd $d$, let
\[
\mathbf R(d,m)=|\mathcal B_d|-r(\mathrm{gen}\,f_d,m).
\]
For every fixed $0<\theta<1/34$, there is a constant $\delta=\delta(\theta)>0$ such that
\[
\sum_{d\leq m^\theta}\widetilde\mu^2(d)4^{\nu(d)}|\mathbf R(d,m)|
\ll m^{1/2-\delta},
\]
where
\[
\widetilde\mu(d)=
\begin{cases}
\mu(d),&2\nmid d,\\
0,&2\mid d.
\end{cases}
\]
\end{lemma}

\begin{proof}
Because of the factor $\widetilde\mu^2(d)$, it is enough to consider square-free odd $d$. By O'Meara \cite[Theorem 102:10]{OM}, the genus of
\(
f_d(\mathbf x)
\)
contains only one spinor genus. Hence
\[
r(\mathrm{spn}\,f_d,m)=r(\mathrm{gen}\,f_d,m).
\]
Let 
\[
\theta(Q,z):=\sum_{{\bf x} \in \Z^3} e(f({\bf x}))=\sum_{m \geq1}r(Q,z)e(mz).
\]
In view of \eqref{e67}, the error $\mathbf R(d,m)$ is bounded, up to the harmless convention of signs and zero coordinates, by the $m$-th Fourier coefficient of
\[
\theta(Q_d,z)-\theta(\mathrm{spn}\,Q_d,z),
\qquad
Q_d=\operatorname{diag}(2,2,2d^2).
\]
Let $N_d$ be the level of $Q_d$, that is, the smallest positive integer such that $N_dQ_d^{-1}$ has integral entries and even diagonal entries. Hence, the level of $Q_d$ is
\[
N_d=4d^2,
\qquad
\det Q_d=8d^2.
\]
We now use the diagonal norm estimate of Waibel, namely Theorem 2 in \cite{Wa}. For
\[
\mathcal F_d(z)=\theta(Q_d,z)-\theta(\mathrm{gen}\,Q_d,z)=\theta(Q_d,z)-\theta(\mathrm{spn}\,Q_d,z),
\]
which gives, by taking $g=\theta(Q_d,z)-\theta(\mathrm{spn}\,Q_d,z)$ in \cite[Theorem 2]{Wa},
\begin{align*}
\langle\mathcal F_d,\mathcal F_d\rangle
&\ll
\left(
\frac{N_d^{3/2}}{F_W(Q_d,3/2)}+
\frac{N_d}{\sqrt{a_3a_2}}
\right)N_d^\varepsilon,
\end{align*}
where $a_1=a_2=2$ and $a_3=2d^2$, and $F_W(Q,\cdot)$ is the genus-invariant factor defined in \cite[(6)]{Wa}. In particular, \cite[Theorem 2]{Wa} implies
$F_W(Q,s/2)\asymp\det Q$ when the greatest common divisor of any $\lfloor s/2\rfloor+1$ of the diagonal entries is bounded by an absolute constant. Since the greatest common divisor of any two diagonal entries of $Q_d$ is at most $2$, Waibel's Theorem 2 gives
\[
F_W(Q_d,3/2)\asymp \det Q_d\asymp d^2.
\]
Consequently
\[
\langle\mathcal F_d,\mathcal F_d\rangle
\ll
\left(\frac{(4d^2)^{3/2}}{d^2}+
\frac{4d^2}{\sqrt{(2d^2)\cdot2}}
\right)d^\varepsilon
\ll d^{1+\varepsilon},
\]
and therefore
\begin{equation}\label{e:diag-norm}
\|\mathcal F_d\|\ll d^{1/2+\varepsilon}.
\end{equation}
Moreover,
$\theta(Q_d,z)-\theta(\mathrm{spn}\,Q_d,z)$ is the component of $\mathcal F_d$ lying in the orthogonal complement of the unary theta subspace, and so it satisfies the same norm bound as in \eqref{e:diag-norm}.

We next insert \eqref{e:diag-norm} into Waibel's coefficient estimate for this orthogonal complement, namely the estimate used in the proof of \cite[Lemma 13]{Wa} by replacing \cite[Theorem 1]{Wa} with \cite[Theorem 2]{Wa} in his proof. This yields
\begin{align}\label{e:Rpointwise-new}
\mathbf R(d,m)
\ll d^{1/2+\varepsilon}
\left(
\frac{m^{13/28}}{N_d^{1/7}}
+
\frac{m^{7/16}}{N_d^{1/16}}
+
 m^{1/4}
\frac{\sqrt{(\widetilde m,N_d^\infty)}\,v^{1/4}\sqrt{(m,N_d)}}{\sqrt{N_d}}
\right)(mN_d)^\varepsilon.
\end{align}
Here $\widetilde m$ and $v$ are the same quantities as in \cite[Lemma 13]{Wa}. Since
\[
(\widetilde m,N_d^\infty)\leq N_d\ll d^2,
\qquad
(m,N_d)\leq N_d\ll d^2,
\qquad
v\leq N_d\ll d^2,
\]
we obtain the pointwise estimate
\begin{equation}\label{e:Rpointwise}
\mathbf R(d,m)
\ll
\left(d^{3/14}m^{13/28}+d^{3/8}m^{7/16}+d^2m^{1/4}\right)m^\varepsilon.
\end{equation}
Finally, $4^{\nu(d)}\ll d^\varepsilon$. Summing \eqref{e:Rpointwise} over $d\leq m^\theta$ gives
\begin{align*}
\sum_{d\leq m^\theta}\widetilde\mu^2(d)4^{\nu(d)}|\mathbf R(d,m)|
&\ll
m^{13/28+17\theta/14+\varepsilon}
+m^{7/16+11\theta/8+\varepsilon}
+m^{1/4+3\theta+\varepsilon}.
\end{align*}
The first exponent is the decisive one, and
\[
13/28+17\theta/14<1/2
\]
is equivalent to $\theta<1/34$. The remaining two inequalities are weaker. Thus, for every fixed $0<\theta<1/34$, choosing $\varepsilon>0$ sufficiently small gives the asserted bound with some $\delta>0$.
\end{proof}

\begin{lemma}\label{lem3}
Let $m$ be a sufficiently large positive integer. Let $\epsilon_1, \epsilon_2>0$ and  $x^{\epsilon_1} \leq z \leq y \leq x^{\frac{1}{2}-\epsilon_2}$,
then there exists some $\epsilon>0$ such that
$$\sum_{z<p\leq y}\sum_{\substack{x_1^2+x_2^2+x_3^2 = m \\ x_1x_2x_3 \equiv 0 \pmod{p^2}}} 1\ll m^{1/2-\varepsilon}.$$
\end{lemma}
\begin{proof}
By symmetry,
\[
\sum_{z<p\leq y}\sum_{\substack{x_1^2+x_2^2+x_3^2 = m \\ x_1x_2x_3 \equiv 0 \, (p^2)}} 1 \ll \sum_{z<p\leq y}\sum_{\substack{x_1^2+x_2^2+x_3^2 = m \\ x_3 \equiv 0 \pmod{p^2}}} 1 + \sum_{z<p\leq y}\sum_{\substack{x_1^2+x_2^2+x_3^2 = m \\ x_1 \equiv x_2 \equiv 0 \pmod{p}}} 1 =: T_1 + T_2.
\]
Since $x_3 \leq m^{1/2}$, the condition $p^2 \mid x_3$ implies that $p \leq m^{1/4}$. Hence
\begin{align}\label{20}
T_1
&\leq
\sum_{z<p\leq m^{1/4}}
\sum_{\substack{x_3\leq m^{1/2}\\ x_3\equiv0\pmod {p^2}}}
\sum_{x_1^2+x_2^2=m-x_3^2}1 \notag\\
&\ll
m^\varepsilon
\sum_{z<p\leq m^{1/4}}
\sum_{\substack{x_3\leq m^{1/2}\\ x_3\equiv0\pmod {p^2}}}1
\ll
m^\varepsilon\left(m^{1/2}z^{-1}+m^{1/4}\right)
\ll m^{1/2-\varepsilon}.
\end{align}
For $T_2$, write $k=a^2+b^2$. Then, for any $\epsilon'>0$,
\begin{align*}
T_2
&=\sum_{z<p\leq y}\sum_{p^2k+c^2=m}r_2(k)\\
&\ll m^{\epsilon'}
\left(
\sum_{\substack{z<p\leq y\\ p\mid m}}
\sum_{\substack{c<m^{1/2}\\ c^2\equiv m\pmod {p^2}}}1
+
\sum_{\substack{z<p\leq y\\ p\nmid m}}
\sum_{\substack{c<m^{1/2}\\ c^2\equiv m\pmod {p^2}}}1
\right),
\end{align*}
where
\begin{align*}
r_{2}(k)=\#\{(a,b)\in \mathbb{Z}^{2}:a^{2}+b^{2}=k\}.
\end{align*}
If $p\nmid m,$ write $$m=pn_{0}+r^{2}\ \ (0<r \leq p-1).$$
Hence $$c=pc_{0}+r\ \ \mbox{or}\ \ c=pc_{0}-r$$ for some non-negative integer $c_{0}.$ Note that $p^{2}|m-c^{2},$ so we have
$$p|n_{0}+2c_{0}r\ \ \mbox{or}\ \ p|n_{0}-2c_{0}r.$$ In either case, there exists a corresponding constant $b\in \Z_p^{*}$ such that $c_{0}\equiv b\pmod p.$ Let $c_{0}=pc_{1}+b,$ then we have
$$c=p(pc_{1}+b)+ r=p^{2}c_{1}+pb+r\ \ \mbox{or}\ \ c=p^{2}c_{1}+pb-r.$$
Therefore
\[
\#\{c \leq m^{1/2}: c^2 \equiv m \pmod{p^2}\} \ll \frac{m^{1/2}}{p^2}+1. 
\]
Thus, for some $\epsilon>0$,
\begin{equation}\label{e62}
\sum_{\substack{z<p\leq y\\ p\nmid m}}
\sum_{\substack{c<m^{1/2}\\ c^2\equiv m\pmod {p^2}}}1
\ll
\sum_{z<p\leq y}\left(\frac{m^{1/2}}{p^2}+1\right)
\ll m^{1/2-\varepsilon}.
\end{equation}
If $p\mid m$, then $c\equiv0\pmod p$. Since $p>z\geq m^{\epsilon_1}$, the number of such prime divisors $p$ of $m$ is $O_{\epsilon_1}(1)$. Therefore
\begin{equation}\label{e63}
\sum_{\substack{z<p\leq y\\ p\mid m}}
\sum_{\substack{c<m^{1/2}\\ c^2\equiv m\pmod {p^2}}}1
\ll
\sum_{\substack{z<p\leq y\\ p\mid m}}
\left(\frac{m^{1/2}}p+1\right)
\ll \frac{m^{1/2}}z+1
\ll m^{1/2-\varepsilon}.
\end{equation}
Combining \eqref{20}, \eqref{e62}, and \eqref{e63} proves the lemma.
\end{proof}

We conclude this section by recalling the sieve lemma used below. Let $\mathcal P$ be a finite set of primes, let $z\geq2$, and let $\mathcal E$ be a finite set of positive integers. Define
\[
P(z)=\prod_{\substack{p\in \mathcal{P} \\ p \leq  z}}p \quad \text{and} \quad S(\mathcal{E},\mathcal{P},z)=\!\!\!\sum_{\substack{t\in\mathcal{E} \\ (t,P(z))=1}}\!\!\!1.
\]
Let $\mathcal E_d=\{t\in\mathcal E:t\equiv0\pmod d\}$. Suppose that
\[
\#\mathcal E_d=\frac{X\omega(d)}d+R_d,
\]
where $\omega(d)$ is multiplicative and there are constants $\kappa,A>0$ such that
\begin{equation}\label{e73}
\prod_{w_1<p\leq w_2}
\left(1-\frac{\omega(p)}p\right)^{-1}
\leq
\left(\frac{\log w_2}{\log w_1}\right)^\kappa
\left(1+\frac A{\log w_1}\right)
\end{equation}
for $2\leq w_1<w_2$.

\begin{lemma}[The sieve lemma]\label{p1}
Suppose that $\kappa\geq1$ and that $2\kappa$ is an integer.
If $(\ref{e73})$ holds and $D$ is a parameter such that $2\le z \le D$, then we have
\begin{equation*}
S(\mathcal{E}, \mathcal{P}, z)
\le X V(z) \Big(
F_{\kappa}\!\Big( \frac{\log D}{\log z} \Big)+ \varepsilon\Big)
+ O\Big(\sum_{\substack{t \mid P(z) \\ t < D}} 4^{\nu(t)} |R_{d}|\Big)
,
\end{equation*}
and
\begin{equation*}
S(\mathcal{E}, \mathcal{P}, z)
\ge X V(z) \Big(f_{\kappa}\!\Big( \frac{\log D}{\log z} \Big)- \varepsilon\Big)
+O\Big( \sum_{\substack{t \mid P(z) \\ t < D}} 4^{\nu(t)} |R_{d}|\Big)
,
\end{equation*}
where
\begin{equation}\label{e:V}
V(z))=\prod_{ p \leq z}\left(1-\frac{\omega(p)}{p}\right),
\end{equation}
$F_{\kappa}(s)$ and $f_{\kappa}(s)$ are the upper and lower bound functions of the higher-dimensional sieve method,
and the constants implied by the $O$-notation depend at most on $\kappa$ and $A$.
\end{lemma}

\noindent{\it Proof}. See \cite[Theorem 9.1]{DH}.

\section{Proofs of the Theorems}\label{S:pv_thms}
\setcounter{lemma}{0}
\setcounter{theorem}{0}
\setcounter{corollary}{0}
\setcounter{remark}{0}
\setcounter{equation}{0}
We shall use a weighted sieve of Richert type, in the form employed by Cai \cite{Cai}. For parameters $y_i$ and $z_i$ $(i=1,2)$ to be chosen below, define
\begin{equation}\label{e42}
g_i(r)=
\sum_{\substack{p\mid r\\ z_i\leq p<y_i}}
\left(1-\frac{\log p}{\log y_i}\right).
\end{equation}
All numerical values of the sieve functions $f_\kappa$ and $F_\kappa$ used below were computed with Galway's Mathematica package \cite{Ga}, as in the standard implementation described in \cite{DH}. In the applications below, $P(z)$ denotes the product of the primes in the relevant sieving set; in particular, in the one-dimensional sieve only odd primes are used.

\begin{proof}[Proof of Theorems \ref{thm1.2} and \ref{thm1.3}]
Let
\[
\eta_1=\frac3{116},
\qquad
D_1=n^{\eta_1-\varepsilon},
\qquad
z_1=D_1^{5/133},
\qquad
y_1=D_1^{529/540}.
\]
Let $0<\theta_1<1$ be chosen below. Consider
\begin{align}\label{e53}
H
&=\sum_{\substack{x_1^2+x_2^2+x_3^2=n\\ (x_1x_2x_3,P(z_1))=1}}
\left(1-\theta_1g_1(x_1x_2x_3)\right) \notag\\
&=H_1-\theta_1H_2.
\end{align}
By Lemmas \ref{p1} and \ref{pro1}, together with \eqref{e:V1}, we have
\begin{align}\label{e54}
H_1
&\geq
XV_1(z_1)\left(f_3\left(\frac{\log D_1}{\log z_1}\right)-\varepsilon\right)
\geq 0.99999XV_1(z_1),
\end{align}
where
\[
V_1(z_1)=\prod_{2<p<z_1}\left(1-\frac{W(p)}p\right)
\gg \frac1{(\log n)^3},
\qquad
X=\frac{\pi\mathfrak S(n)n^{1/2}}4\gg n^{1/2-\varepsilon}.
\]
Similarly,
\begin{align}\label{e55}
H_2
&=\sum_{z_1\leq p<y_1}
\left(1-\frac{\log p}{\log y_1}\right)
\sum_{\substack{x_1^2+x_2^2+x_3^2=n\\ (x_1x_2x_3,P(z_1))=1\\ x_1x_2x_3\equiv0\pmod p}}1 \notag\\
&\leq
XV_1(z_1)
\sum_{z_1\leq p<y_1}
\frac{W(p)}p
\left(1-\frac{\log p}{\log y_1}\right)
\left(F_3\left(\frac{\log(D_1/p)}{\log z_1}\right)+\varepsilon\right) \notag\\
&\leq
3XV_1(z_1)
\left(
\int_{5/133}^{529/540}
\left(1-\frac{540}{529}t\right)
\frac{F_3(26.6(1-t))}{t}\,dt+\varepsilon
\right).
\end{align}
The last step is the standard partial-summation step in the weighted sieve, using \eqref{e:V1}; compare \cite[Lemma 4.1]{Ir}. The numerical computation gives
\begin{equation}\label{e56}
\int_{5/133}^{529/540}
\left(1-\frac{540}{529}t\right)
\frac{F_3(26.6(1-t))}{t}\,dt
\leq 2.62214.
\end{equation}
Choosing $\theta_1=0.12712$ and combining \eqref{e53}--\eqref{e56}, we obtain
\begin{equation}\label{e57}
H\geq0.000010689\,XV_1(z_1).
\end{equation}

Let $H^+$ be the partial sum of $H$ over those triples for which
\[
1-\theta_1g_1(x_1x_2x_3)>0.
\]
Since each summand in $H^+$ is at most $1$, \eqref{e57} implies that the number of such triples is
\begin{equation}\label{e58}
\gg XV_1(z_1)\gg n^{1/2-\varepsilon}.
\end{equation}
By Lemma \ref{lem3}, the contribution of triples for which $p^2\mid x_1x_2x_3$ for some prime $z_1\leq p<y_1$ is $O(n^{1/2-\delta})$ for some $\delta>0$. Hence, after decreasing $\varepsilon$ if necessary, there are still $\gg n^{1/2-\varepsilon}$ triples counted by $H^+$ for which
\begin{equation}\label{squarefree-medium}
p^2\nmid x_1x_2x_3\qquad(z_1\leq p<y_1).
\end{equation}
For any such triple, the positivity of the weight gives
\begin{align*}
\Omega(x_1x_2x_3)
&\leq
\sum_{\substack{z_1\leq p<y_1\\ p\mid x_1x_2x_3}}1
+
\sum_{\substack{p\geq y_1\\ p^a\parallel x_1x_2x_3}}\frac{\log p}{\log y_1} \\
&<
\sum_{\substack{z_1\leq p<y_1\\ p\mid x_1x_2x_3}}
\frac{\log p}{\log y_1}
+\frac1{\theta_1}
+
\sum_{\substack{p\geq y_1\\ p^a\parallel x_1x_2x_3}}
\frac{\log p}{\log y_1} \\
&\leq
\frac{\log(x_1x_2x_3)}{\log y_1}+\frac1{\theta_1}
\leq
\frac1{\theta_1}+\frac{810}{529}(\eta_1-\varepsilon)^{-1}
<67.07264.
\end{align*}
Since $\Omega(x_1x_2x_3)$ is an integer, this proves Theorem \ref{thm1.2}.

The same triples also prove Theorem \ref{thm1.3}. Indeed, for $1\leq i\leq3$ we have
\[
g_1(x_i)\leq g_1(x_1x_2x_3),
\]
and, using \eqref{squarefree-medium} as above,
\begin{align*}
\Omega(x_i)
&<
\frac{\log x_i}{\log y_1}+\frac1{\theta_1}
\leq
\frac1{\theta_1}+\frac{270}{529}(\eta_1-\varepsilon)^{-1}
<27.60194.
\end{align*}
Thus each $x_i$ is a $P_{27}$-number, and \eqref{e58} gives the asserted lower bound for the number of such representations.
\end{proof}

\begin{proof}[Proof of Theorem \ref{thm1.1}]
Let
\[
\eta_2=\frac1{34},
\qquad
D_2=m^{\eta_2-\varepsilon},
\qquad
z_2=D_2^{10/51},
\qquad
y_2=D_2^{54/55}.
\]
Let $0<\theta_2<1$ be chosen below, and recall that
\[
g_2(u)=\sum_{\substack{p\mid u\\ z_2\leq p<y_2}}
\left(1-\frac{\log p}{\log y_2}\right).
\]
Consider the weighted sifted sum
\begin{align}\label{e15}
\Sigma
&=\sum_{\substack{x_1^2+x_2^2+x_3^2=m\\ (x_3,P(z_2))=1}}
\left(1-\theta_2g_2(x_3)\right) \notag\\
&=\Sigma_1-\theta_2\Sigma_2.
\end{align}
Here $\Sigma_1$ counts representations with no prime divisor less than $z_2$ in $x_3$, while $\Sigma_2$ subtracts a controlled weight from those representations for which $x_3$ has a prime divisor in $[z_2,y_2)$.
The product estimate \eqref{e71} shows that this is a sieve of dimension one. Applying Lemma \ref{p1} with Lemma \ref{lem1} gives
\begin{align}\label{e16}
\Sigma_1
&\geq
r_3(m)V_2(z_2)
\left(f_1\left(\frac{\log D_2}{\log z_2}\right)-\varepsilon\right),
\end{align}
where
\[
V_2(z_2)=\prod_{2<p<z_2}
\left(1-\frac{\omega(p,m)}p\right)
\gg\frac1{\log m}
\]
by \eqref{e71}. Since
\[
\frac{\log D_2}{\log z_2}=\frac{51}{10}=5.1,
\]
the numerical values of the linear sieve functions give
\begin{equation}\label{e16a}
\Sigma_1\geq0.99865\,r_3(m)V_2(z_2).
\end{equation}
For $\Sigma_2$ we first expose the prime $p\mid x_3$ in the range $z_2\leq p<y_2$ and then sieve the remaining condition $(x_3,P(z_2))=1$. Lemmas \ref{p1} and \ref{lem1} give
\begin{align}\label{e17}
\Sigma_2
&=\sum_{z_2\leq p<y_2}
\left(1-\frac{\log p}{\log y_2}\right)
\sum_{\substack{x_1^2+x_2^2+x_3^2=m\\ (x_3,P(z_2))=1\\ x_3\equiv0\pmod p}}1 \notag\\
&\leq
r_3(m)V_2(z_2)
\sum_{z_2\leq p<y_2}
\frac{\omega(p,m)}p
\left(1-\frac{\log p}{\log y_2}\right)
\left(F_1\left(\frac{\log(D_2/p)}{\log z_2}\right)+\varepsilon\right)+O(m^{1/2-\delta}).
\end{align}
Writing $p=D_2^t$, we have
\[
\frac{\log(D_2/p)}{\log z_2}=\frac{1-t}{10/51}=5.1(1-t),
\qquad
1-\frac{\log p}{\log y_2}=1-\frac{55}{54}t.
\]
Using partial summation together with \eqref{e71}, the last prime sum is bounded by
\begin{align}\label{e17a}
\Sigma_2
&\leq
r_3(m)V_2(z_2)
\left(
\int_{10/51}^{54/55}
\left(1-\frac{55}{54}t\right)
\frac{F_1(5.1(1-t))}{t}\,dt
+\varepsilon
\right).
\end{align}
The numerical computation of the sieve functions gives
\begin{equation}\label{e18}
\int_{10/51}^{54/55}
\left(1-\frac{55}{54}t\right)
\frac{F_1(5.1(1-t))}{t}\,dt
\leq1.11531.
\end{equation}
Choose
\[
\theta_2=0.89540.
\]
Combining \eqref{e15}, \eqref{e16a}, \eqref{e17a}, and \eqref{e18}, and taking $\varepsilon>0$ sufficiently small, gives
\begin{equation}\label{e19}
\Sigma
\geq
\left(0.99865-0.89540\cdot1.11531+O(\varepsilon)\right)r_3(m)V_2(z_2)
\geq0.000001426\,r_3(m)V_2(z_2).
\end{equation}
Let $\Sigma^+$ be the part of $\Sigma$ supported on triples satisfying
\[
1-\theta_2g_2(x_3)>0.
\]
Since every positive summand is at most $1$, \eqref{e19}, Siegel's lower bound for $r_3(m)$, and $V_2(z_2)\gg1/\log m$ imply
\begin{equation}\label{21}
\#\{\text{triples counted by }\Sigma^+\}
\gg r_3(m)V_2(z_2)
\gg m^{1/2-\varepsilon}.
\end{equation}
We next discard the triples for which $p^2\mid x_3$ for some prime $z_2\leq p\leq m^{1/4}$. By Lemma \ref{lem3}, this discarded set has cardinality $O(m^{1/2-\delta})$ for some $\delta>0$. Hence there remain $\gg m^{1/2-\varepsilon}$ triples for which
\[
(x_3,P(z_2))=1,
\qquad
p^2\nmid x_3\quad(z_2\leq p\leq m^{1/4}),
\qquad
1-\theta_2g_2(x_3)>0.
\]
Write
\[
x_3=2^az,
\qquad
2\nmid z.
\]
Then $(z,P(z_2))=1$. Moreover, the preceding square-factor condition implies that the odd integer $z$ is square-free: indeed, if $p^2\mid z$, then $p\geq z_2$, while $p>m^{1/4}$ is impossible because $p^2\leq z\leq x_3\leq m^{1/2}$.

For the remaining triples, the inequality $1-\theta_2g_2(x_3)>0$ gives
\[
\sum_{\substack{z_2\leq p<y_2\\p\mid z}}1
<
\frac1{\theta_2}
+
\sum_{\substack{z_2\leq p<y_2\\p\mid z}}
\frac{\log p}{\log y_2}.
\]
Since $z$ is square-free and has no prime divisor below $z_2$, it follows that
\begin{align}\label{e64}
\Omega(z)
&=\sum_{\substack{p\geq z_2\\ p\mid z}}1 \notag\\
&<
\sum_{\substack{z_2\leq p<y_2\\p\mid z}}
\frac{\log p}{\log y_2}
+\frac1{\theta_2}
+
\sum_{\substack{p\geq y_2\\p\mid z}}
\frac{\log p}{\log y_2} \notag\\
&\leq
\frac{\log z}{\log y_2}+\frac1{\theta_2}
\leq
\frac1{\theta_2}+\frac{55}{108}(\eta_2-\varepsilon)^{-1}
<18.5.
\end{align}
In the penultimate inequality we used $z\leq x_3\leq m^{1/2}$ and
$\log y_2=(54/55)(\eta_2-\varepsilon)\log m$. Since $\Omega(z)$ is an integer, \eqref{e64} gives $\Omega(z)\leq18$. Thus $z$ is a $P_{18}$-number. The lower bound \eqref{21} gives the asserted number of representations, and the theorem follows.
\end{proof}


\begin{thebibliography}{KMP}
\bibitem{Ban} S. Banerjee, {\it On a conjecture of Sun about sums of restricted squares,} J. Number Theory,  {\bf256} (2024), 253-289.
\bibitem{Blo} V. Blomer, {\it Ternary quadratic forms, and sums of three squares with restricted variables,} CRM Proc. Lect. Notes, {\bf46} (2008), 1-17.
\bibitem{BB} V. Blomer and J. Br\"udern, {\it A three squares theorem with almost primes,} Bull. London Math. Soc., {\bf37} (2005), 507-513.
\bibitem{Br} J. Br\"udern and E. Fouvry, {\it Lagrange's four squares theorem with almost prime variables,} J. reine angew Math., {\bf454} (1994), 59-96.
\bibitem{Cai2} Y. Cai, {\it Lagrange's four squares theorem with variables of special type}, Int. J. Number Theory {\bf6} (2010), 1801-1817.
\bibitem{Cai} Y. Cai, {\it Gauss's three squares theorem involving almost-primes,} Rocky Mountain J. Math., {\bf 42} (2012), no. 4, 1115-1134.
\bibitem{DH} H. G. Diamond and H. Halberstam, {\it A Higher-Dimensional Sieve Method,} Cambridge Tracts in Mathematics, {\bf177} (Cambridge University Press, Cambridge, 2008), with an appendix `Procedures for computing sieve functions' by W. F. Galway.
\bibitem{Ga} W. Galway, http://www.math.uiuc.edu/SieveTheoryBook/SieveFunctions.m.
\bibitem{Gre} G. Greaves, {\it On the representation of a number in the form $x^2+y^2+p^2+q^2$ where $p,q$ are odd primes}, Acta Arith. {\bf29} (1976), 257-274.
\bibitem{HR} H. Halberstam, D. R. Heath-Brown and H. E. Richert, {\it Almost--primes in short intervals,} in: Recent Progress in Analytic Number Theory \uppercase\expandafter{\romannumeral1}, Academic Press, 1981, 69-103.
\bibitem{HT} D. R. Heath-Brown and D. I. Tolev, {\it Lagrange's four squares theorem with one prime and three almost prime variables}, J. Reine Angew. Math. {\bf558} (2003), 159-224.
\bibitem{Ir} A. J. Irving, {\it Almost-prime values of polynomials at prime arguments,} Bull. Lond. Math. Soc., {\bf47} (2015), 593-606.
\bibitem{Iw1} H. Iwaniec, {\it Rosser's sieve,} Acta Arith., {\bf36} (1980) 171-202.
\bibitem{Iw2} H. Iwaniec, {\it A new form of the error term in the linear sieve,} Acta Arith., {\bf37} (1980), 307-320.
\bibitem{Lv} G. L\"u, {\it Gauss's three squares theorem with almost prime variables,} Acta. Arith., {\bf128} (2007), 391-399.
\bibitem{OM} O. T. O'Meara, {\it Introduction to Quadratic Forms}, Springer, 1973.
\bibitem{Sun2}  Z.-W. Sun, {\it New Conjectures in Number Theory and Combinatorics,} (in Chinese), Harbin Institute
of Technology Press, Harbin, 2021.
\bibitem{Sun0} Z.-W. Sun, Sequence A308661 at OEIS (On-Line Encyclopedia of Integer Sequences), http://oeis.org/A308661.
\bibitem{Ri} H.-E. Richert, {\it Selberg's sieve with weights}, Mathematika, {\bf16} (1969), 1-22.
\bibitem{Si} C. Siegel, {\it \"Uber die Klassenzahl algebraischer Zahlenk\"orper,} Acta Arith., {\bf1} (1935) 83-86.
\bibitem{Wa} F. Waibel, {\it Uniform bounds for norms of theta series and arithmetic applications}, Math. Proc. Cambridge Phil. Soc., {\bf173} (2022), 660-691.
\end{thebibliography}
\end{document}